\documentclass[reqno]{amsart}
\usepackage{hyperref}

\def\N{{\mathbb{N}}}

\def\R{{\mathbb{R}}}

\begin{document}
\title[Pontryagin principles]{Lightenings of assumptions for Pontryagin principles in infinite horizon and discrete time}

\author[BLOT and NGO]
{Jo\"{e}l BLOT and Thoi Nhan NGO}

\address{Jo\"{e}l Blot: Laboratoire SAMM EA 4543,\newline
Universit\'{e} Paris 1 Panth\'{e}on-Sorbonne, centre P.M.F.,\newline
90 rue de Tolbiac, 75634 Paris cedex 13,
France.}
\email{blot@univ-paris1.fr}
\address{Thoi Nhan Ngo: Laboratoire SAMM EA 4543,\newline
Universit\'{e} Paris 1 Panth\'{e}on-Sorbonne, centre P.M.F.,\newline
90 rue de Tolbiac, 75634 Paris cedex 13,
France.}
\email{ngothoinhan@gmail.com}
\date{January 22, 2016}
\begin{abstract} In the infinite-horizon and discrete-time framework we establish maximum principles of Pontryagin under assumptions which weaker than these ones of existing results. We avoid several assumptions of continuity and of Fr\'echet-differentiability and of linear independence.
\end{abstract}
\maketitle
\numberwithin{equation}{section}
\newtheorem{theorem}{Theorem}[section]
\newtheorem{lemma}[theorem]{Lemma}
\newtheorem{example}[theorem]{Example}
\newtheorem{remark}[theorem]{Remark}
\newtheorem{definition}[theorem]{Definition}
\newtheorem{corollary}[theorem]{Corollary}
\newtheorem{proposition}[theorem]{Proposition}

\section{Introduction}
The aim of this paper is to establish maximum principles of Pontryagin under assumptions which are weaker than those of existing results. Now we state the considered problems.
\vskip2mm
For all $t \in \N$ let $X_t$ be a nonempty open subset of $\R^n$, $U_t$ be a nonempty subset of $\R^d$, and $f_t : X_t \times U_t \rightarrow X_{t+1}$ be a mapping. We introduce the two following dynamical systems.
\begin{itemize}
\item[(Di)] $x_{t +1} \leq f_t(x_t, u_t)$, $ t \in \N$.
\item[(De)] $x_{t +1} = f_t(x_t, u_t)$, $ t \in \N$.
\end{itemize}
The order in (Di) is the usual order of $\R^n$: when $x = (x^1,...,x^n)$ and $y = (y^1,...,y^n)$ belong to $\R^n$, $x \leq y$ means $x^i \leq y^i$ for all $i \in \{ 1,..., n \}$. We fix $\sigma \in X_0$, and when $k \in \{ i,e \}$, we define ${\rm Adm}_k$ as the set of the $(\underline{x}, \underline{u}) = ((x_t)_{t \in \N}, (u_t)_{t \in \N}) \in \prod_{t \in \N}X_t \times \prod_{ t \in \N} U_t$ such that $(\underline{x}, \underline{u})$ satisfies (Dk) for all $t \in \N$ and such that $x_0 = \sigma$.
\vskip1mm
For all $t \in \N$, we consider the function $\phi_t : X_t \times U_t \rightarrow \R$. When $k \in \{ i,e \}$, we define ${\rm Dom}_k$ as the set of the $(\underline{x}, \underline{u}) \in {\rm Adm}_k$ such that the series $\sum_{t=0}^{+ \infty} \phi_t(x_t,u_t)$ is convergent in $\R$. We define the functional $J : {\rm Dom}_k \rightarrow \R$ by setting $J(\underline{x}, \underline{u}) := \sum_{t=0}^{+ \infty} \phi_t(x_t,u_t)$.
\vskip1mm
When $k \in \{ i,e \}$, we consider the following list of problems.
\begin{itemize}
\item[$(P^1_k)$] Maximize $J(\underline{x}, \underline{u})$ when $(\underline{x}, \underline{u}) \in {\rm Dom}_k$.
\item[$(P^2_k)$] Find $(\underline{\hat{x}}, \underline{\hat{u}}) \in {\rm Adm}_k$ such that, for all $(\underline{x}, \underline{u}) \in {\rm Adm}_k$,\\
$\limsup_{h \rightarrow + \infty}( \sum_{t=0}^h \phi_t(\hat{x}_t, \hat{u}_t) - \sum_{t=0}^h \phi_t(x_t,u_t)) \geq 0$.
\item[$(P^3_k)$] Find $(\underline{\hat{x}}, \underline{\hat{u}}) \in {\rm Adm}_k$ such that, for all $(\underline{x}, \underline{u}) \in {\rm Adm}_k$,\\
$\liminf_{h \rightarrow + \infty}( \sum_{t=0}^h \phi_t(\hat{x}_t, \hat{u}_t) - \sum_{t=0}^h \phi_t(x_t,u_t)) \geq 0$.
\end{itemize}
Now we describe the contents of the paper.
\vskip1mm
In Section 2 we specify notions of differentiability and their notation, and we recall the method of reduction to finite horizon (Theorem \ref{th21}).\\
\indent
In Section 3 we establish weak maximum principles where the values of the optimal control belong to the interior of the sets of controls for system which governed by difference inequations (Theorem \ref{th31} and Theorem \ref{th32}). These results are new and only use the G\^ateaux differentiability of the criterion, of the vector field and of the inequality constraints. Neither continuity nor Fr\'echet differentiability is necessary. These principles use recent results on multipliers rules in static optimization which are established in \cite{Bl4}\\
\indent
In Section 4, we establish a weak maximum principle when the sets of controls are defined by inequalities (Theorem \ref{th43}) when the system is governed by difference inequations. This result also only uses the G\^ateaux differentiability of the criterion, of the vector field and of the inequality constraints and a condition of separation of the origine and of the convex hull of the G\^ateaux differentials of the inequelities constraints in the spirit of a Mangarasian-Fromowitz condition. Secondly we establish a weak maximum principle when the sets of controls are define by equalities and inequalities (Theorem \ref{th47}) when the system is governed by a difference inequation. Such a case is treated in \cite{Bl3} (Theorem 3.1 and Theorem 3.2). In comparison with the result of \cite{Bl3}, the improvements are the following ones: we avoid a condition of continuity for the saturated inequality constraints and for the vector field, we avoid a condition of linear independence of all the differentials of the constraints. A similar result is Theorem \ref{th48} for which the system is governed by a difference equation.
\section{Notation and Recall}
When $E$ and $F$ are finite-dimensional real normed vector spaces, when $A \subset E$, when $\Phi : A \rightarrow F$ is a mapping, and $a \in A$, $\Phi$ is said G\^ateaux differentiable at $a$ when, for all $v \in E$, $\vec{D} \Phi(a,v) := \lim_{s \rightarrow 0} \frac{1}{s}(\Phi(a + s v) - \Phi(a))$ exists for all $v \in E$ and when $v \mapsto \vec{D} \Phi(a,v)$ is linear, then its G\^ateaux differential of $\Phi$ at $a$ is $D_G \Phi(a) \in {\mathcal L}(E,F)$ defined by $D_G \Phi(a)v := \vec{D}\Phi(a,v)$. When it exists, the Fr\'echet differential of $\phi$ at $a$ is denoted by $D \phi(a)$. When $E = E_1 \times E_2$ is a product of normed vector spaces, $a = (a_1,a_2)$, $D_{G,1} \Phi(a_1,a_2)$ (respectively $D_{G,2} \Phi(a_1,a_2)$) is the G\^ateaux differential of $\Phi( \cdot, a_2)$ at $a_1$ (respectively $\Phi(a_1, \cdot)$  at $a_2$) and when $\Phi$ is Fr\'echet differentiable at $(a_1,a_2)$, $D_1 \Phi(a_1,a_2)$ (respectively $D_2 \phi(a_1,a_2)$ denotes the Fr\'echet differential of $\Phi( \cdot, a_2)$ at $a_1$ (respectively $D_2 \Phi(a_1, \cdot)$ at $a_2$). 
\vskip1mm
The method of reduction to finite horizon, which comes from \cite{BCh}, is contained in the following result.
\begin{theorem}\label{th21} The two assertions hold. 
\begin{itemize}
\item[(a)] Let $(\underline{\hat{x}}, \underline{\hat{u}})$ be a solution of $(P^j_i)$ when $j \in \{ 1,2,3 \}$. Then, for all $h \in \N_*$, $(\hat{x}_0,..., \hat{x}_{h+1}, \hat{u}_0,..., \hat{u}_h)$ is a solution of the following finite-horizon problem
\[ (F^h_i) 
\left\{ 
\begin{array}{cl}
{\rm Maximize} & \sum_{t=0}^h \phi_t(x_t, u_t)\\
{\rm when} & (x_0,...,x_{h+1},u_0,...,u_h) \in \prod_{t=0}^{h+1} X_t \times \prod_{t=0}^h U_t\\
\null & x_{t+1} \leq f_t(x_t,u_t) \; {\rm when} \; t \in \{0, ...,h \}\\
\null & x_0 = \sigma, x_{h+1} = \hat{x}_{h+1}.
\end{array}
\right.
\]
\item[(b)] Let $(\underline{\hat{x}}, \underline{\hat{u}})$ be a solution of $(P^j_e)$ when $j \in \{ 1,2,3 \}$. Then, for all $h \in \N_*$, $(\hat{x}_0,..., \hat{x}_{h+1}, \hat{u}_0,..., \hat{u}_h)$ is a solution of the following finite-horizon problem
\[ (F^h_e) 
\left\{ 
\begin{array}{cl}
{\rm Maximize} & \sum_{t=0}^h \phi_t(x_t, u_t)\\
{\rm when} & (x_0,...,x_{h+1},u_0,...,u_h) \in \prod_{t=0}^{h+1} X_t \times \prod_{t=0}^h U_t\\
\null & x_{t+1} = f_t(x_t,u_t) \; {\rm when} \; t \in \{0, ...,h \}\\
\null & x_0 = \sigma, x_{h+1} = \hat{x}_{h+1}.
\end{array}
\right.
\]
\end{itemize}
\end{theorem}
The proof of this theorem is given in \cite{BCh} and in \cite{BH}. Note that this result does not require any special assumption.
\section{Weak Pontryagin principles with interior optimal controls}
In this section we consider the case where values of the optimal control sequence belong to the topological interior of the set $U_t$ of the considered controls at each time $t$, and where the system is governed by the difference inequation (Di).
\begin{theorem}\label{th31} Let $(\underline{\hat{x}}, \underline{\hat{u}})$ be a solution of $(P^j_i)$ when $j \in \{ 1,2,3 \}$. We assume that the following assumptions are fulfilled.
\begin{itemize}
\item[(i)] For all $t \in \N$, $\hat{u}_t \in {\rm int}U_t$.
\item[(ii)]  For all $t \in \N$, $\phi_t$ and $f_t$ are G\^ateaux differentiable at $(\hat{x}_t, \hat{u}_t)$.
\item[(iii)]  For all $t \in \N$, for all $\alpha \in \{ 1,...,n \}$, $f^{\alpha}_t$ is lower semicontinuous at $(\hat{x}_t, \hat{u}_t)$ when $f^{\alpha}_t(\hat{x}_t, \hat{u}_t) > \hat{x}_{t+1}^{\alpha}$.
\item[(iv)] For all $t \in \N$, $D_{G,1}f_t(\hat{x}_t, \hat{u}_t)$ is invertible.
\end{itemize}
Then there exist $\lambda_0 \in \R$ and $ (p_t)_{t \in \N_*} \in (\R^{n*})^{\N_*}$ which satisfy the following properties.
\begin{itemize}
\item[(NN)] $(\lambda_0, p_1) \neq (0,0)$.
\item[(Si)] $\lambda_0 \geq 0$ and, for all $t \in \N_*$, $p_t \geq 0$.
\item[(S${\ell})$] For all $t \in \N$, for all $\alpha \in \{1, ...,n \}$, $p^{\alpha}_{t+1} \cdot( f^{\alpha}_t(\hat{x}_t, \hat{u}_t) - \hat{x}^{\alpha}_{t+ 1}) = 0$.
\item[(AE)] For all $t \in \N_*$, $p_t = p_{t+1} \circ D_{G,1}f_t(\hat{x}_t, \hat{u}_t) + \lambda_0 D_{G,1} \phi_t(\hat{x}_t, \hat{u}_t)$.
\item[(WM)] For all $t \in \N$, $ p_{t+1} \circ D_{G,2}f_t(\hat{x}_t, \hat{u}_t) + \lambda_0 D_{G,2} \phi_t(\hat{x}_t, \hat{u}_t) = 0$. 
\end{itemize}
\end{theorem}
\begin{proof} Using Theorem \ref{th21} we can assert that, for all $h \in \N_*$, $(\hat{x}_0, ..., \hat{x}_{h+1}, \hat{u}_0,..., \hat{u}_h)$ is a solution of $(F^h_i)$.
\vskip1mm
We introduce the function $\psi : \prod_{t = 1}^h X_t \times \prod_{t = 0}^h U_t \rightarrow \R$ by setting 
$$\psi(x_1,...,x_h, u_0,...,u_h) := \phi_0(\sigma, u_0) + \sum_{t=1}^h \phi_t(x_t,u_t).$$
We introduce the mapping $\psi_0 : \prod_{t = 1}^h X_t \times \prod_{t=0}^h U_t \rightarrow \R^n$ by setting\\
$\psi_0(x_1,...,x_h, u_0,...,u_h) := f_0(\sigma, u_0) - x_1$. For all $t \in \{ 1,..., h-1 \}$ we introduce the mapping $\psi_t : \prod_{t = 1}^h X_t \times \prod_{t=0}^h U_t \rightarrow \R^n$ by setting
$$\psi_t(x_1,...,x_h, u_0,...,u_h) := f_t(x_t, u_t) - x_{t+1}.$$ 
We introduce the mapping $\psi_h : \prod_{t = 1}^h X_t \times \prod_{t=0}^h U_t \rightarrow \R^n$ by setting 
$$\psi_h(x_1,...,x_h, u_0,...,u_h) := f_h(x_h, u_h) - \hat{x}_{h+1}.$$
Then we can formulate $(F^h_i)$ in the following form.
\begin{equation}\label{eq31}
\left.
\begin{array}{cl}
{\rm Maximize} & \psi(x_1,...,x_h, u_0,...,u_h)\\
{\rm when} & \forall t \in \{1,...,h \}, x_t \in X_t, \;  u_t \in U_t \\
\null & \forall  t \in \{0,...,h \}, \forall \alpha \in \{1,...,n \}, \psi^{\alpha}_t(x_1,...,x_h, u_0,...,u_h) \geq 0
\end{array}
\right\}
\end{equation}
where the $\psi^{\alpha}_t$ are the coordinates of $\psi_t$.\\
Our assumptions (i, ii, iii) imply that the assumptions of Theorem 3.1 in \cite{Bl4} are fulfilled and so we know that, for all $h \in \N_*$, there exists $(\lambda_0^h, p_1^h, ..., p^h_{h+1}) \in \R \times (\R^{n*})^{h+1}$ which satisfies the following conditions.
\begin{equation}\label{eq32}
(\lambda_0^h, p_1^h, ..., p^h_{h+1})  \neq (0,0,...,0).
\end{equation}
\begin{equation}\label{eq33}
\lambda_0^h \geq 0, {\rm and} \; \forall t \in \{1,..., h+1 \},  p_t^h \geq 0.
\end{equation}
\begin{equation}\label{eq34}
\forall t \in \{0,...,h \}, \forall \alpha \in \{1,...,n \}, p_{t+1,\alpha}^h \cdot (f_t^{\alpha}(\hat{x}_t, \hat{u}_t) - \hat{x}^{\alpha}_{t+1}) = 0.
\end{equation}
\begin{equation}\label{eq35}
\forall t \in \{1,...,h \}, p_t^h = p_{t+1}^h \circ D_{G,1}f_t(\hat{x}_t, \hat{u}_t) + \lambda_0^h D_{G,1} \phi_t(\hat{x}_t, \hat{u}_t).
\end{equation}
\begin{equation}\label{eq36}
\forall t \in \{1,...,h \},  p_{t+1}^h \circ D_{G,2}f_t(\hat{x}_t, \hat{u}_t) + \lambda_0^h D_{G,2} \phi_t(\hat{x}_t, \hat{u}_t) = 0.
\end{equation}
Using assumption (iv) we can formulate (\ref{eq35}) as follows.
\begin{equation}\label{eq37}
\forall t \in \{1,...,h \}, p^h_{t+1} = (p_t^h - \lambda_0^h D_{G,1} \phi_t(\hat{x}_t, \hat{u}_t)) \circ  D_{G,1}f_t(\hat{x}_t, \hat{u}_t)^{-1}.
\end{equation}
From this last equation we easily see that $(\lambda_0^h, p_1^h) = (0,0) \Longrightarrow (\lambda_0^h, p_i^h, ..., p^h_{h+1})  = (0,0,...,0)$ and then from (\ref{eq32}) we can assert that 
\begin{equation}\label{eq38}
(\lambda_0^h, p_1^h) \neq (0,0).
\end{equation}
Since the set of the lists of multipliers of Problem (\ref{eq31}) is a cone, we can normalize the multipliers by setting 
\begin{equation}\label{eq39}
\vert \lambda_0^h \vert + \Vert p_1^h \Vert = 1.
\end{equation}
Since the values of the sequence $(\lambda_0^h, p_1^h)_{h \in \N_*}$ belong to the unit sphere of $\R \times \R^{n*}$ which is compact, using the Bolzano-Weierstrass theorem we can say that there exist an increasing function $\varphi : \N_* \rightarrow \N_*$ and $(\lambda_0, p_1) \in \R \times \R^{n*}$ such that $\vert \lambda_0 \vert + \Vert p_1 \Vert = 1$, $\lim_{h \rightarrow + \infty} \lambda^{\varphi (h)}_0 = \lambda_0$ and $\lim_{h \rightarrow + \infty} p^{\varphi (h)}_1 = p_1$. 
\vskip1mm
Note that $p_2^{\varphi (h)} = (p_1^{\varphi (h)} - \lambda_0^{\varphi (h)} D_{G,1} \phi_1(\hat{x}_1, \hat{u}_1))\circ D_{G,1} f_1(\hat{x}_1, \hat{u}_1)^{-1}$ for all $h \geq t-1$, which implies that 
$$p_2 := \lim_{h \rightarrow + \infty} p^{\varphi (h)}_2 = (p_1 - \lambda_0  D_{G,1} \phi_1(\hat{x}_1, \hat{u}_1))\circ D_{G,1} f_1(\hat{x}_1, \hat{u}_1)^{-1}.$$
 Proceeding recursively we define, for all $t \in \N_*$, $p_{t+1} := \lim_{h \rightarrow + \infty} p^{\varphi (h)}_{t+1} =$ \\ $ \lim_{h \rightarrow + \infty} (p_t^{\varphi(h)} - \lambda_0^{\varphi(h)}  D_{G,1} \phi_t(\hat{x}_t, \hat{u}_t))\circ D_{G,1} f_t(\hat{x}_t, \hat{u}_t)^{-1} = 
(p_t - \lambda_0  D_{G,1} \phi_t(\hat{x}_t, \hat{u}_t))\circ D_{G,1} f_t(\hat{x}_t, \hat{u}_t)^{-1}$. And so we have built $\lambda_0 \in \R$ and a sequence $(p_t)_{t \in \N_*} \in (\R^{n*})^{\N_*}$ which satisfies (AE).
\vskip1mm
We have yet seen that (NN) is satisfied. From (\ref{eq33}) we obtain (Si). From (\ref{eq34}) we obtain (S${\ell}$). From (\ref{eq36}) we obtain (WM).
\end{proof}
\begin{theorem}\label{th32}
Let $(\underline{\hat{x}}, \underline{\hat{u}})$ be a solution of $(P^j_i)$ with $j \in \{ 1,2,3 \}$. We assume that the assumption (i,ii,iii) of Theorem \ref{th31} are fulfilled. Moreover we assume that the following assumption is fulfilled.
\begin{itemize}
\item[(v)] For all $t \in \N_*$, for all $\alpha, \beta \in \{ 1,...,n \}$, $\frac{\partial f^{\alpha}_t(\hat{x}_t, \hat{u}_t)}{\partial x^{\beta}} \geq 0$ and for all $\alpha \in \{ 1,...,n \}$, $\frac{\partial f^{\alpha}_t(\hat{x}_t, \hat{u}_t)}{\partial x^{\alpha}} > 0$.
\end{itemize}
Then the conclusions of Theorem \ref{th31} hold.
\end{theorem}
\begin{proof}
Proceeding as in the proof of Theorem \ref{th31} we obtain that, for all 
$h \in \N_*$, $(\hat{x}_0,...,\hat{x}_{h+1}, \hat{u}_0,...,\hat{u}_h)$ is a solution of problem (\ref{eq31}) and conditions (\ref{eq32} - \ref{eq37}) are fulfilled.
\vskip1mm
For all $t \in \N_*$ we define $\gamma_t := \min_{1 \leq \alpha \leq n}  \frac{\partial f^{\alpha}_t(\hat{x}_t, \hat{u}_t)}{\partial x^{\alpha}}  \in (0, + \infty)$.
\vskip1mm
Under assumption (v), when $v \in \R^n$, for all $\alpha \in \{1,...,n \}$, we have 
$$(D_{G,1} f_t(\hat{x}_t, \hat{u}_t) \cdot v)^{\alpha} = \sum_{\beta = 1}^n \frac{\partial f^{\alpha}_t(\hat{x}_t, \hat{u}_t)}{\partial x^{\beta}} v^{\beta} \geq \frac{\partial f^{\alpha}_t(\hat{x}_t, \hat{u}_t)}{\partial x^{\alpha}} v^{\alpha} \geq \gamma_t  v^{\alpha}$$
which implies $D_{G,1} f_t(\hat{x}_t, \hat{u}_t) \cdot v \geq \gamma_t v$. Then using Lemma 2.3 in \cite{BH} (p. 37) we can assert that, for all $\pi \in (\R^{n*})_+$, $\pi \circ D_{G,1} f_t(\hat{x}_t, \hat{u}_t) \geq \gamma_t \pi$. Then we have, for all $t \in \{1,...,h \}$,
$$p_t^h  =  p^h_{t + 1} \circ D_{G,1} f_t(\hat{x}_t, \hat{u}_t) + \lambda_0^h D_{G,1} \phi_t (\hat{x}_t, \hat{u}_t)
 \geq  \gamma_t  p^h_{t + 1} +  \lambda_0^h D_{G,1} \phi_t (\hat{x}_t, \hat{u}_t)$$
which implies $0 \leq p^h_{t + 1} \leq \frac{1}{\gamma_t} (p^h_t - \lambda_0^h D_{G,1} \phi_t (\hat{x}_t, \hat{u}_t)$ which implies
$$\Vert 
p^h_{t + 1} \Vert \leq \frac{1}{\gamma_t} (\Vert p^h_t \Vert + \lambda_0^h \Vert D_{G,1} \phi_t (\hat{x}_t, \hat{u}_t) \Vert)$$
 which implies, since from (\ref{eq33}) and (\ref{eq39}) we have $\lambda_0^h = \vert \lambda_0^h \vert \leq 1$, the following relation holds for all $t \in \N_*$ and for all $h \geq t-1$.
\begin{equation}\label{eq310}
\Vert  p^h_{t + 1} \Vert \leq \frac{1}{\gamma_t} (\Vert p^h_t \Vert + \Vert D_{G,1} \phi_t (\hat{x}_t, \hat{u}_t) \Vert).
\end{equation}
Now we want to prove the following assertion.
\begin{equation}\label{eq311}
\forall t \in \N_*, \exists \zeta_t \in (0, + \infty), \forall h \geq t-1, \;  \Vert p^h_t \Vert \leq \zeta_t.
\end{equation}
We proceed by induction. When $t= 1$, from (\ref{eq39}) we know that $ \Vert p^h_t \Vert \leq 1$, and so it suffices to take $\zeta_1 := 1$. We assume that (\ref{eq311}) holds for $t$, then for $t+1$, from (\ref{eq310}) we obtain
$$\Vert  p^h_{t + 1} \Vert \leq \frac{1}{\gamma_t} ( \zeta_t + \Vert D_{G,1} \phi_t (\hat{x}_t, \hat{u}_t) \Vert) =: \zeta_{t+1}$$
and so (\ref{eq311}) is proven.
\vskip1mm
Using (\ref{eq311}) and the diagonal process of Cantor as it is formulated in \cite{BH} (Theorem A.1, p. 94), we can assert that there exists an increasing function $\rho : \N_* \rightarrow \N_*$ and a sequence $(p_t)_{t \in \N_*} \in (\R^{n*}_+)^{\N_*}$ such that, for all $t \in \N_*$, $\lim_{h \rightarrow + \infty} p^{\rho(h)}_t = p_t$. Now we conclude as in the proof of Theorem \ref{th31}.
\end{proof} 
\section{Weak Pontryagin principles with constrained controls}
In this section we first consider the case where the sets of controls are defined by inequalities for each $t \in \N$.
\begin{equation}\label{eq41}
U_t = \bigcap_{1 \leq k \leq m} \{ u \in \R^d : g^k_t(u) \geq 0 \}
\end{equation}
where $g^k_t : \R^d \rightarrow \R$.
\begin{lemma}\label{lem41}
Let $E$ be a finite-dimensional real normed vector space and $I$ be a nonempty finite set.
Let $(\varphi_i)_{i \in I} \in (E^*)^I$. The three following assertions are equivalent.
\begin{itemize}
\item[(i)] $0 \notin co \{ \varphi_i : i \in I \}$.
\item[(ii)] For all $(\lambda_i)_{i \in I} \in (\R_+)^I$, $\sum_{i \in I} \lambda_i \varphi_i = 0 \Longrightarrow$  $ \lambda_i = 0$ for all $i \in I$.
\item[(iii)] There exists $w \in E$ such that, $\langle \varphi_i, w \rangle > 0$  for all $i \in I$.
\end{itemize}
\end{lemma}
\begin{proof}
First we prove that non(ii) implies non(i).
From non(ii) we deduce that there exists $(\lambda_i)^{i \in I} \in (\R_+)^I$ such that $(\lambda_i)^{i \in I} \neq 0$ and $\sum_{i \in I} (\frac{\lambda_i}{\sum_{j \in I} \lambda_j}) \varphi_i = 0$ which implies  non(i). Secondly we prove that non(i) implies non(ii). From non(i) there exists $(\alpha_i)_{i \in I} \in \R_+Î$ such that $\sum_{i \in I} \alpha_i = 1$ and $0 = \sum_{i \in I} \alpha_i \varphi_i$, and since  $(\alpha_i)_{i \in I}$ is non zero, non(ii) is fulfilled. And so we have proven that non(i) and non(ii) are equivalent.
\vskip1mm
To prove that (i) implies (iii), note that $0 \notin co \{ \varphi_i : i \in I \} =: K$, and $K$ is a nonempty convex compact set. Using the theorem of separation of Hahn-Banach, we can assert that there exist $\xi \in \R^{n**}$ and $a \in (0, + \infty)$ such that $\langle \xi, \varphi \rangle \geq a$ for all $\varphi \in K$, and $\langle \xi, 0 \rangle = 0 < a$. Since $\R^n$ is reflexive, there exists $w \in \R^n$ such $\langle \xi, \varphi \rangle = \langle \varphi, w \rangle$ for all $\varphi \in \R^{n*}$. Therefore for all $i \in I$, we have $\langle \varphi_i, w \rangle \geq a > 0$ that is (iii).
\vskip1mm
To prove that (iii) implies (i) we set $\gamma := \min_{i \in I} \langle \varphi_i, w \rangle > 0$. When $\varphi \in co \{ \varphi_i : i \in I \}$, there exists $(\alpha_i)_{i \in I} \in \R_+Î$ such that $\sum_{i \in I} \alpha_i = 1$ and $\varphi = \sum_{i \in I} \alpha_i \varphi_i$. Then we have $\langle \varphi, w \rangle = \sum_{i \in I} \alpha_i \langle \varphi_i, w \rangle \geq \sum_{i \in I} \alpha_i \gamma = \gamma > 0$ which implies $\varphi \neq 0$, and so (i) is satisfied.  
\end{proof}
\begin{lemma}\label{lem42}
Let $E$ be a finite-dimensional real normed vector space and $I$ be a nonempty finite set.
Let $(\varphi_i)_{i \in I} \in (E^*)^I$ such that $0 \notin co \{ \varphi_i : i \in I \}$. For all $i \in I$, let $(r_i^h)_{h \in \N_*} \in R_+^{\N_*}$. We assume that the sequence $(\psi_h)_{h \in \N_*} := ( \sum_{i \in I} r_i^h \varphi_i)_{h \in \N_*}$ is bounded in $E^*$. \\
Then there exists an increasing function $\rho : \N_* \rightarrow \N_*$ such that, for all $i \in I$, the sequence $(r_i^{\rho(h)})_{h \in \N_*}$ is convergent in $R_+$.
\end{lemma}
\begin{proof}
First we prove that $\liminf_{h \rightarrow + \infty} \sum_{i \in I} r_i^h < + \infty$. We proceed by contradiction: we assume that $\liminf_{h \rightarrow + \infty} \sum_{i \in I} r_i^h = + \infty$. Therefore we have\\
 $\lim_{h \rightarrow + \infty} \sum_{i \in I} r_i^h = + \infty$. We set $s_i^h := \frac{r_i^h}{\sum_{j \in I} r_j^h} \in \R_+$. We have $\sum_{i \in I} s_i^h = 1$ and therefore $\sum_{i \in I} s_i^h \varphi_i \in co \{ \varphi_i : i \in I \}$. Note that $\Vert \sum_{i \in I} s_i^h \varphi_i \Vert = \frac{1}{\sum_{j \in I} r_j^h} \Vert \psi_h \Vert$ converges to $0$ when $h \rightarrow + \infty$ since $(\psi_h)_{h \in \N_*}$ is bounded. Therefore we have $\lim_{ h \rightarrow + \infty} \sum_{i \in I} s_i^h \varphi_i = 0$ which implies that $0 \in co \{ \varphi_i : i \in I \}$ that is a contradiction with one assumption. And so we have proven that  $s := \liminf_{h \rightarrow + \infty} \sum_{i \in I} r_i^h < + \infty$.
\vskip1mm
Now we can assert that there exists an increasing function $\tau : \N_* \rightarrow \N_*$ such that $\lim_{h \rightarrow + \infty} \sum_{i \in I} r^{\tau(h)}_i = s$. Therefore there exists $M \in \R_+$ such that $0 \leq \sum_{i \in I} r^{\tau(h)}_i \leq M$ for all $h \in \N_*$. 
Since for all $i \in I$, we have $0 \leq r^{\tau(h)}_i \leq \sum_{j \in I} r^{\tau(h)}_j \leq M$, i.e. the sequence $(r^{\tau(h)}_i)_{h \in \N_*}$ is bounded in $R_+$. Using several times the Bolzano-Weierstrass theorem we can assert that there exist an increasing function $\tau_1 : \N_* \rightarrow \N_*$ and $r^*_i \in \R_+$ for all $i \in I$, such that $\lim_{h \rightarrow + \infty} r_i^{\tau \circ \tau_1(h)} = r_i^*$. It suffices to take $\rho := \tau \circ \tau_1$. 
\end{proof}
\begin{theorem}\label{th43}
Let $(\underline{\hat{x}}, \underline{\hat{u}})$ be a solution of $(P^j_i)$ where $j \in \{ 1,2,3 \}$ and where the sets $U_t$ are defined by (\ref{eq41}). We assume that the following assumptions are fulfilled.
\begin{itemize}
\item[(i)] For all $t \in \N$, $\phi_t$ and $f_t$ are G\^ateaux differentiable at $(\hat{x}_t, \hat{u}_t)$.
\item[(ii)] For all $t \in \N$, for all $k \in \{1,...,m \}$, $g_t^k$ is G\^ateaux differentiable at $\hat{u}_t$.
\item[(iii)] For all $t \in \N$, for all $\alpha \in \{1,...,n \}$, $f_t^{\alpha}$ is lower semicontinuous at $(\hat{x}_t, \hat{u}_t)$ when $f_t^{\alpha}(\hat{x}_t, \hat{u}_t) > \hat{x}_{t+1}^{\alpha}$.
\item[(iv)] For all $t \in \N$, for all $k \in \{1,...,m \}$, $g^k_t$ is lower semicontinuous at $\hat{u}_t$ when $g^k_t(\hat{u}_t) > 0$.
\item[(v)] For all $t \in \N$, $0 \notin co \{ D_Gg_t^k(\hat{u}_t) : k \in I_t^s \}$ where $I_t^s := \{ k \in \{ 1,...,m \} : g_t^k(\hat{u}_t) = 0 \}$.
\item[(vi)] For all $t \in \N_*$, $D_{G,1}f_t(\hat{x}_t, \hat{u}_t)$ is invertible.
\item[(vii)]For all $t \in \N_*$ for all $\alpha, \beta \in \{1,...,n \}$, $\frac{\partial f^{\alpha}_t(\hat{x}_t, \hat{u}_t)}{\partial x^{\beta}} \geq 0$ and for all $\alpha \in \{1,...,n \}$, $\frac{\partial f^{\alpha}_t(\hat{x}_t, \hat{u}_t)}{\partial x^{\alpha}} > 0$.
\end{itemize}
Then, under (i-vi) or under (i-v) and (vii), there exist $\lambda_0 \in \R$, $(p_t)_{t \in \N_*} \in (\R^{n*})^{\N_*}$, $(\mu_t^1)_{t \in \N} \in \R^{\N}$, ..., and $(\mu_t^m)_{t \in \N} \in \R^{\N}$ which satisfy the following conditions. 
\begin{itemize}
\item[(NN)] $(\lambda_0, p_1) \neq (0,0)$.
\item[(Si)] $\lambda_0 \geq 0$, , $p_t \geq 0$ for all $t \in \N_*$, and $\mu_t^k \geq 0$ for all $t \in \N$ and for all $k \in \{1,...,m \}$.
\item[(S${\ell})$] For all $t \in \N$, for all $\alpha \in \{1, ...,n \}$, $p^{\alpha}_{t+1} \cdot( f^{\alpha}_t(\hat{x}_t, \hat{u}_t) - \hat{x}^{\alpha}_{t+ 1}) = 0$, and for all $k \in \{1,...,m \}$, $\mu_t^k \cdot g^k_t(\hat{u}_t) = 0$.
\item[(AE)] For all $t \in \N_*$, $p_t = p_{t+1} \circ D_{G,1}f_t(\hat{x}_t, \hat{u}_t) + \lambda_0 D_{G,1} \phi_t(\hat{x}_t, \hat{u}_t)$.
\item[(WM)] For all $t \in \N$, \\
$ p_{t+1} \circ D_{G,2}f_t(\hat{x}_t, \hat{u}_t) + \lambda_0 D_{G,2} \phi_t(\hat{x}_t, \hat{u}_t)  + \sum_{k=1}^m \mu_t^k D_Gg^k_t(\hat{u}_t) = 0.$ 
\end{itemize}
\end{theorem}
\begin{proof}
From Theorem \ref{th21} (a) we know that, for all $h \in \N_*$, $(\hat{x}_0, ..., \hat{x}_{h+1}, \hat{u}_0, ..., \hat{u}_h)$ is a solution of the following finite-horizon problem.
\[
\left\{
\begin{array}{cl}
{\rm Maximize} & J(x_0,...,x_{h+1}, u_0,...,u_h) =  \sum_{t=0}^h \phi_t(x_t,u_t)\\
{\rm when} & \forall t \in \{ 0, ...,h \}, f_t(x_t,u_t) - x_{t+1} \geq 0\\
\null & \forall t \in  \{ 0, ...,h \}, x_t \in X_t\\
\null & x_0 = \sigma, x_{h+1} = \hat{x}_{h+1}\\
\null &  \forall t \in  \{ 0, ...,h \}, \forall k \in \{1,...,m \}, g^k_t(u_t) \geq 0.
\end{array}
\right.
\]
From Theorem 3.1 in \cite{Bl4} we can assert that there exists\\
 $(\lambda^h_0, p_1^h, ..., p_{h+1}^h, \mu_1^{1,h}, ..., \mu_h^{m,h}) \in \R \times (\R^{n*})^h \times \R^{mh}$ which satisfies the following assertions.
\begin{equation}\label{eq42}
(\lambda^h_0, p_1^h, ..., p_{h+1}^h, \mu_1^{1,h}, ..., \mu_m^{h,h}) \neq 0.
\end{equation}
\begin{equation}\label{eq43}
\left.
\begin{array}{l}
\lambda_0^h \geq 0, \forall t \in \{1,...,h+1 \}, p_t^h \geq 0,\\
 {\rm and} \; \forall t \in \{0, ..., h \}, \forall k \in \{ 1,...,m \}, \mu_t^{k,t} \geq 0.
\end{array}
\right\}
\end{equation}
\begin{equation}\label{eq44}
\left.
\begin{array}{l}
\forall t \in \{1,...,h+1 \}, \forall \alpha \in \{1, ...,n \}, p^{\alpha}_{t+1} \cdot( f^{\alpha}_t(\hat{x}_t, \hat{u}_t) - \hat{x}^{\alpha}_{t+ 1}) = 0,\\ {\rm and} \; \forall t \in \{0, ..., h \}, \forall k \in \{1,...,m \}, \mu_t^{k,h} \cdot g^k_t(\hat{u}_t) = 0.
\end{array}
\right\}
\end{equation}
\begin{equation}\label{eq45}
\left.
\begin{array}{l} 
\forall t \in \{1,...,h \},\\
 p_t^h = p_{t+1}^h \circ D_{G,1}f_t(\hat{x}_t, \hat{u}_t) + \lambda_0^h D_{G,1} \phi_t(\hat{x}_t, \hat{u}_t).
\end{array}
\right\}
\end{equation}
\begin{equation}\label{eq46}
\left.
\begin{array}{l} 
\forall t \in \{0,...,h \},\\
 p_{t+1}^h \circ D_{G,2}f_t(\hat{x}_t, \hat{u}_t) + \lambda_0^h D_{G,2} \phi_t(\hat{x}_t, \hat{u}_t) 
+\sum_{k=1}^m \mu_t^{k,h} D_Gg^k_t(\hat{u}_t) = 0.
\end{array}
\right\}
\end{equation} 
Using (\ref{eq45}) under (vi) or (vii) and working as in the proof of Theorem \ref{th31} or Theorem \ref{th32}, we obtain
$$(\lambda^h_0, p_1^h) = (0,0) \Longrightarrow (\lambda^h_0, p_1^h,...,p_{h+1}^h) = (0,0,...,0).$$
Proceeding by contradiction, assuming that $(\lambda^h_0, p_1^h) = (0,0) $, from the previous implication and (\ref{eq46}) we obtain $\sum_{k=1}^m \mu_t^{k,h} D_Gg^k_t(\hat{u}_t) = 0$, and then using Lemma \ref{lem41} we obtain that the $\mu_t^{k,h} = 0$. Therefore we obtain a contradiction with (\ref{eq42}). And so we have proven that $(\lambda^h_0, p_1^h) \neq (0,0)$. Under (vi) proceeding as in the proof of Theorem \ref{th31} and under (vii) proceeding as in the proof of Theorem \ref{th32} we obtain the existence of an increasing function $\rho : \N_* \rightarrow \N_*$ and of $\lambda_0 \in \R_+$ and of $(p_t)_{t \in \N_*} \in (\R^{n*}_+)^{\N_*}$  such that $\lambda_0 = \lim_{h \rightarrow + \infty} \lambda_0^{\rho(h)}$, $p_t = \lim_{h \rightarrow + \infty} p_t^{\rho(h)}$, $(\lambda_0, p_1) \neq (0,0)$, and $p_t^{\alpha} \cdot (f^{\alpha}_t(\hat{x}_t, \hat{u}_t) - \hat{x}^{\alpha}_{t+1}) = 0$ for all $t \in \N_*$ and for all $\alpha \in \{1,...,n \}$.
\vskip1mm
We fix $t \in \N$ and we consider, for all $h \in \N_*$, 
\[
\begin{array}{ccl}
\varphi_h & := & \sum_{k \in I^s_t} \mu_t^{k, \rho(h)} D_G g^k_t(\hat{u}_t) = \sum_{k=1}^m \mu_t^{k, \rho(h)} D_G g^k_t(\hat{u}_t) \\
\null & = & - ( p_{t+1}^{\rho(h)} \circ D_{G,2}f_t(\hat{x}_t, \hat{u}_t) + \lambda_0^{\rho(h)} D_{G,2} \phi_t(\hat{x}_t, \hat{u}_t)) .
\end{array}
\]
Therefore we have $\lim_{h \rightarrow + \infty} \varphi_h = - ( p_{t+1} \circ D_{G,2}f_t(\hat{x}_t, \hat{u}_t) + \lambda_0 D_{G,2} \phi_t(\hat{x}_t, \hat{u}_t)) $, and consequently the sequence $(\varphi_h)_{h \in \N_*}$ is bounded in $\R^{n*}$. Using Lemma \ref{lem42} we can assert that exist an increasing function $\rho_1 : \N_* \rightarrow \N_*$ and $\mu_t^1$, ..., $\mu_t^m \in \R_+$ such that $\lim_{h \rightarrow + \infty} \mu_t^{k, \rho \circ \rho_1(h)} = \mu_t^k \in \R_+$. And then the assertions (NN), (Si), (S${\ell}$), (AE) and (WM) are satisfied. 
\end{proof}
\vskip2mm
Now we consider the case where the sets of controls are defined by equalities and inequalities for each $t \in \N$,
\begin{equation}\label{eq47}
U_t =  \bigcap_{1 \leq k \leq m_i} \{ u \in \R^d : g^k_t(u) \geq 0 \} \cap ( \bigcap_{1 \leq k \leq m_e } \{ u \in \R^d : e_t^k(u) = 0 \})
\end{equation}
where $e_t^k : \R^d \rightarrow \R$.
\begin{lemma}\label{lem44}
Let $E$ be a real finite-dimensional normed vector space; let $J$ and $K$ be two nonempty finite sets, and let $(\psi^j)_{j \in J}$ and $(\varphi^k)_{k \in K}$ be two families of elements of the dual $E^*$. Then the two following assertions are equivalent.
\begin{itemize}
\item[(i)] $span \{ \psi^j : j \in J \} \cap  co \{\varphi^k : k \in K \} = \emptyset$.
\item[(ii)] There exists $w \in E$ such that $\langle \psi^j, w \rangle = 0$ for all $j \in J$ and $\langle \varphi^k, w \rangle  > 0$ for all $k \in K$.
\end{itemize}
\end{lemma}
\begin{proof} We set $S := span \{ \psi^j : j \in J \} $ and $C :=  co \{\varphi^k : k \in K \}$.\\
$[ i \Longrightarrow ii ]$  Under (i) using the theorem of separation of Hahn-Banach, there exist $\xi \in E^{**}$ and $a \in (0, + \infty)$ such that $\langle \xi, \psi \rangle \leq a$ for all $\psi \in S$, and $\langle \xi, \varphi \rangle > a$ for all $\varphi \in C$. When $\psi \in S$ is non zero, we have $\vert \langle \xi, \psi \rangle \vert  \leq a$ since $- \psi \in S$, and therefore, for all $\lambda \in \R$, we have $\vert \lambda \vert \cdot \vert \langle \xi, \psi \rangle \vert  \leq a$ which is impossible if $\vert \langle \xi, \psi \rangle \vert  \neq 0$, therefore we have $\langle \xi, \psi \rangle = 0$ for all $\psi \in S$. Since $E^{**}$ is isomorphic to $E$ there exists $w \in E$ such $\langle \xi, \chi \rangle = \langle \chi, w \rangle$ for all $\chi \in E^*$, and then we obtain (ii).\\
$[ ii \Longrightarrow i ]$ Under (ii) we define $a := \min_{k \in K} \langle \varphi^k, w \rangle > 0$. When $\varphi \in C$ there exists $(\theta_k)_{k \in K} \in \R_+^K$ such that $\sum_{k \in K} \theta_k = 1$ and $\sum_{k \in K} \theta_k \varphi^k = \varphi$. Then $\langle \varphi, w \rangle = \sum_{k \in K} \theta_k \langle \varphi^k, w \rangle \geq  \sum_{k \in K} \theta_k  \cdot a > 0$. When $\psi \in S$ there exists $(\zeta_j)_{j \in J} \in \R^J$ such that $\sum_{j \in J} \zeta_j \psi^j = \psi$. Therefore we have $\langle \psi, w \rangle =  
\sum_{j \in J} \zeta_j \langle \psi^j, w \rangle  = 0$. We have proven that $\langle \psi, w \rangle = 0$ for all $\psi \in S$ and $\langle \varphi, w \rangle > 0$ for all $\varphi \in C$, which implies (i).
\end{proof}

\begin{lemma}\label{lem45}
In the framework of Lemma \ref{lem44}, under condition (i) of Lemma \ref{lem44}, when $(\lambda_j)_{j \in J} \in \R^J$ and $(\mu_k)_{k \in K} \in \R_+^K$, we have 

$$\sum_{j \in J} \lambda_j \psi^j + \sum_{k \in K} \mu_k \varphi^k = 0 \Longrightarrow (\forall k \in K,\;  \mu_k = 0).$$
\end{lemma}
\begin{proof} We proceed by contraposition, we assume that there exists $k \in K$ such that $\mu_k \neq 0$. Then $\overline{\mu} := \sum_{k \in K} \mu_k > 0$ and so $\sum_{k \in K} \frac{\mu_k}{\overline{\mu}} \varphi^k \in co \{ \varphi^k : k \in K \}$ and $\sum_{k \in K} \frac{\mu_k}{\overline{\mu}} \varphi^k = -\sum_{j \in J}\frac{\lambda_j}{\overline{\mu}} \psi^j \in span\{\psi^j : j \in J \}$ which provides a contradiction
with condition (i).
\end{proof}
\begin{lemma}\label{lem46}
Let $E$ be a real finite-dimensional normed vector space; let $J$ and $K$ be two nonempty finite sets, and let $(\psi^j)_{j \in J}$ and $(\varphi^k)_{k \in K}$ be two families of elements of the dual $E^*$. We assume that the following assumptions are fulfilled.
\begin{itemize}
\item[(a)] The family $(\psi^j)_{j \in J}$ is linearly independent.
\item[(b)] $span \{ \psi^j : j \in J \} \cap co \{ \varphi^k : k \in K \} = \emptyset$.
\end{itemize}
Let $(\lambda^h_j)_{j \in J} \in \R^J$ and $(\mu^h_k)_{k \in K} \in \R_+^K$ for all $h \in \N_*$ such that the sequence $(\chi^h)_{h \in \N_*} := (\sum_{j \in J} \lambda^h_j \psi^j + \sum_{k \in K} \mu^h_k \varphi^k)_{h \in \N_*}$ is bounded in $E^*$. Then there exists an increasing function $\rho : \N_* \rightarrow \N_*$ such that the sequences $(\lambda^{\rho(h)}_j)_{h \in \N_*}$ are convergent in $\R$ for all $j \in J$ and the sequences and $(\mu^{\rho(h)}_k)_{h \in \N_*}$ are convergent in $\R_+$ for all $k \in K$.
\end{lemma}
\begin{proof}
We set $S := span \{ \psi^j : j \in J \}$ and $C :=  co \{ \varphi^k : k \in K \}$. First we prove that $\liminf_{ h \rightarrow + \infty} \sum_{k \in K} \mu^h_k < + \infty$. We proceed by contradiction, we assume that $\liminf_{ h \rightarrow + \infty} \sum_{k \in K} \mu^h_k = + \infty$. Therefore we have $s := \lim_{h \rightarrow + \infty} \sum_{k \in K} \mu^h_k = + \infty$. We set $\pi^h_k := \frac{\mu^h_k}{\sum_{k' \in K} \mu^h_{k'}} \in \R_+$. We have $\sum_{k \in K} \pi^h_k = 1$, and therefore $\sum_{k \in K} \pi^h_k \varphi^k \in C$. Note that
$$\Vert \sum_{j \in J} \frac{\lambda^h_j}{\sum_{k' \in K} \mu^h_{k'}} \psi^j + \sum_{k \in K} \pi^h_k \varphi^k \Vert = \frac{1}{\sum_{k' \in K} \mu^h_{k'}} \Vert \chi^h \Vert \rightarrow 0$$
 when $h \rightarrow + \infty$, therefore 
$$\lim_{h \rightarrow + \infty} (\sum_{j \in J} \frac{\lambda^h_j}{\sum_{k' \in K} \mu^h_{k'}} \psi^j + \sum_{k \in K} \pi^h_k \varphi^k) = 0.$$
 Since $C$ is compact there exists an increasing function $\tau : \N_* \rightarrow \N_*$ and $\varphi_* \in C$ such that $\lim_{h \rightarrow + \infty} \sum_{k \in K} \pi^{\tau(h)}_k \varphi^k = \varphi_*$.  Consequently $\lim_{h \rightarrow + \infty} \sum_{j \in J} \frac{ - \lambda^{\tau(h)}_j}{\sum_{k' \in K} \mu^{\tau(h)}_{k'}} \psi^j = \varphi_*$. Since a finite-dimensional normed vector space is complete, $S$ is closed in $E^*$, and consequently we have $\varphi_* \in S$, and then $\varphi_* \in S \cap C$ which is a contradiction with assumption (b). And so we have proven that  $\liminf_{ h \rightarrow + \infty} \sum_{k \in K} \mu^h_k < + \infty$. Therefore there there exists an increasing function $r : \N_* \rightarrow \N_*$ such that $\lim_{h \rightarrow + \infty} \sum_{k \in K} \mu^{r(h)}_k = \liminf_{ h \rightarrow + \infty} \sum_{k \in K} \mu^h_k $. Therefore the sequence \\
$(\sum_{k \in K} \mu^{r(h)}_k)_{h \in \N_*})_{h \in \N_*}$ is bounded in $\R_+$. Since $0 \leq \mu^{r(h)}_k \leq \sum_{k \in K}\mu^{r(h)}_k$, we obtain that $(\mu^{r(h)}_k)_{h \in \N_*}$ is bounded in $\R_+$ for all $k \in K$. Therefore $(\sum_{k \in K} \mu^{r(h)}_k \varphi^k)_{h \in \N_*}$ is bounded in $E^*$. Therefore $(\sum_{j \in J} \lambda^{r(h)}_j \psi^j)_{h \in \N_*} = (\chi^{r(h)} - \sum_{k \in K} \mu^{r(h)}_k \varphi^k)_{h \in \N_*}$ is bounded as a difference of two bounded sequences. Under assumption (a) we can use Lemma 5.5 in \cite{Bl3} and assert that there exists an increasing function $r_1 : \N_* \rightarrow \N_*$ such $(\lambda^{r \circ r_1(h)}_j)_{h \in \N_*}$ is convergent in $\R$ for all $j \in J$. Using card$K$ times the Bolzano-Weierstrass theorem, there exists an increasing function $r_2 : \N_* \rightarrow \N_*$ such that $(\mu^{r \circ r_1 \circ r_2(h)}_k)_{h \in \N_*}$ is convergent in $\R_+$. Taking $\rho := r \circ r_1 \circ r_2$ we have proven the lemma.
\end{proof}
\begin{theorem}\label{th47}
Let $(\underline{\hat{x}}, \underline{\hat{u}})$ be a solution of $(P^j_i)$ where $j \in \{1,2,3 \}$ and where the sets $U_t$ are defined in (\ref{eq47}). We assume that the following assumptions are fulfilled for all $t \in \N$.
\begin{itemize}
\item[(i)] $\phi_t$ is Fr\'echet differentiable at $(\hat{x}_t, \hat{u}_t)$.
\item[(ii)] For all $\alpha \in \{ 1,..., n \}$, $f^{\alpha}_t$ is Fr\'echet differentiable at $(\hat{x}_t, \hat{u}_t)$ when\\
 $f^{\alpha}_t(\hat{x}_t, \hat{u}_t) = x^{\alpha}_{t+1}$.
\item[(iii)] For all $\alpha \in \{ 1,..., n \}$, $f^{\alpha}_t$ is lower semicontinuous and G\^ateaux differentiable at ($\hat{x}_t, \hat{u}_t)$ when $f^{\alpha}_t(\hat{x}_t, \hat{u}_t) > x^{\alpha}_{t+1}$.
\item[(iv)] For all $k \in \{ 1, ..., m_i \}$, $g^k_t$ is Fr\'echet differentiable at $\hat{u}_t$ when $g^k_t(\hat{u}_t) = 0$.
\item[(v)] For all $k \in \{ 1, ..., m_i \}$, $g^k_t$ is lower semicontinuous and G\^ateaux differentiable at $\hat{u}_t$ when $g^k_t(\hat{u}_t) > 0$.
\item[(vi)] For all $j \in \{ 1, ..., m_e \}$, $e^j_t$ is continuous on a neighborhood of $\hat{u}_t$ and Fr\'echet differentiable at $\hat{u}_t$.
\item[(vii)]$span \{ De^j_t(\hat{u}_t) : j \in \{ 1,..., m_e \} \} \cap co \{ D_Gg^k_t(\hat{u}_t) : k \in I^s_t \} = \emptyset$, where \\
$I^s_t := \{ k \in \{1,...,m_i \} : g^k_t(\hat{u}_t) = 0 \}$.
\item[(viii)] $De_t^1(\hat{u}_t)$, ..., $De^{m_e}(\hat{u}_t)$ are linearly independent.
\item[(ix)] For all $t \in \N_*$, $D_{G,1}f_t(\hat{x}_t, \hat{u}_t)$ is invertible.
\item[(x)]For all $t \in \N_*$ for all $\alpha, \beta \in \{1,...,n \}$, $\frac{\partial f^{\alpha}_t(\hat{x}_t, \hat{u}_t)}{\partial x^{\beta}} \geq 0$ and for all $\alpha \in \{1,...,n \}$, $\frac{\partial f^{\alpha}_t(\hat{x}_t, \hat{u}_t)}{\partial x^{\alpha}} > 0$.
\end{itemize}
Then under (i-ix) or under (i-viii) and (x) there exist $\lambda_0 \in \R$, $(p_t)_{t \in \N_*} \in (R^{n*})^{\N_*}$, $(\lambda_{1,t})_{t \in \N} \in \R^{\N}$, ..., $(\lambda_{m_e, t})_{ t \in \N} \in \R^{\N}$, $(\mu_{1,t})_{t \in \N} \in \R^{\N}$, ..., $(\mu_{m_i, t})_{ t \in \N}\in \R^{\N}$ which satisfy the following conditions.
\begin{itemize}
\item[(NN)] $(\lambda_0, p_1) \neq (0,0)$.
\item[(Si)] $\lambda_0 \geq 0$, $p_t \geq 0$ for all $t \in \N_*$, $\mu_{k,t} \geq 0$ for all $t \in \N$ and for all $k \in \{1,...,m_i \}$.
\item[(S${\ell}$)] For all $t \in \N$, for all $\alpha \in \{1,...,n \}$, $p_{t+1}^{\alpha} \cdot (f_t^{\alpha}(\hat{x}_t, \hat{u}_t) - x_{t+1}^{\alpha}) = 0$, and for all $k \in \{1,...,m_i \}$, $\mu_{k,t} \cdot g_t^k(\hat{u}_t) = 0$.
\item[(AE)] For all $t \in \N_*$, $p_t = p_{t+1} \circ D_{G,1} f_t(\hat{x}_t, \hat{u}_t) + \lambda_0 D_1 \phi_t (\hat{x}_t, \hat{u}_t)$.
\item[(WM)] For all $t \in \N$, \\
$p_{t+1} \circ D_{G,2} f_t(\hat{x}_t, \hat{u}_t) + \lambda_0 D_2 \phi_t (\hat{x}_t, \hat{u}_t) + \displaystyle\sum_{j=1}^{m_e}\lambda_{j,t} De^j(\hat{u}_t) + \displaystyle\sum_{k=1}^{m_i} \mu_{k,t} D_G g^k_t(\hat{u}_t) = 0.$
\end{itemize}
\end{theorem}
\begin{proof}
From Theorem \ref{th21} (a) we know that, for all $h \in \N_*$, $(\hat{x}_0, ..., \hat{x}_{h+1}, \hat{u}_0,..., \hat{u}_h)$ is a solution of the following finite-horizon problem. 
\[
\left\{
\begin{array}{cl}
{\rm Maximize} & \sum_{t = 0}^h \phi_t(x_t,u_t)\\
{\rm when} & \forall t \in \{0,...,h \}, f_t(x_t,u_t) - x_{t+1} \geq 0\\
\null & \forall t \in \{0, ..., h+1 \}, x_t \in X_t\\
\null & x_0 = \sigma, x_{h+1} = \hat{x}_{h+1}\\
\null & \forall t \in \{ 0, ..., h \}, \forall j \in \{1,...,m_e \}, e^j_t(u_t) = 0\\
\null & \forall t \in \{ 0, ...,h \}, \forall k \in \{1,...,m_i \}, g^k_t(u_t) \geq 0.
\end{array}
\right.
\]
We introduce the following elements
\[
\begin{array}{l}
{\bf z} := (x_1,...,x_h, u_0, ...,u_h)\\
\Phi({\bf z}) := \sum_{t = 0}^h \phi_t(x_t,u_t)\\
F^{\alpha}_0({\bf z}) := f^{\alpha}_0(\sigma, u_0) - x^{\alpha}_1, 1 \leq \alpha \leq n\\
F^{\alpha}_t({\bf z}) := f^{\alpha}_t(x_t,u_t) - x^{\alpha}_t, 1 \leq t \leq h-1,  1 \leq \alpha \leq n\\
F^{\alpha}_h({\bf z}) := f^{\alpha}_h(x_h,u_h) - \hat{x}_{h+1}^{\alpha}, 1 \leq \alpha \leq n\\
E^j_t({\bf z}) := e^j_t(u_t), 0 \leq t \leq h, 1 \leq j \leq m_e\\
G^k_t({\bf z}) := g^k_t(u_t),  0 \leq t \leq h, 1 \leq k \leq m_i.
\end{array}
\]
Then the previous optimization problem can be written as follows.
\begin{equation}\label{eq48}
\left.
\begin{array}{cl}
{\rm Maximize} & \Phi({\bf z})\\
{\rm when} & \forall t \in \{0,...,h \}, \forall \alpha \in \{1,...,n \}, F^{\alpha}_t({\bf z}) \geq 0\\
\null & \forall t \in \{0,...,h \}, \forall j \in \{1,...,m_e \}, E^j_t({\bf z}) = 0\\
\null & \forall t \in \{0,...,h \}, \forall k \in \{ 1,...,m_i \}, G^k_t({\bf z}) \geq 0.
\end{array}
\right\}
\end{equation}
We see that our assumptions (i-vi) imply that the assumptions of Theorem 3.2 in \cite{Bl4} are fulfilled and consequently we obtain the existence of real numbers $\lambda_0^h$, $p_{t, \alpha}^h$ (for $t \in \{1,...,h+1 \}$ and $\alpha \in \{1,...,n \}$), $\lambda^h_{t,j}$ (for $t \in \{0,...,h \}$ and $j \in \{ 1,...,m_e \}$), $\mu^h_{t,k}$ (for $t \in \{0,...,h \}$ and $k \in \{1,..., m_i \}$) which satisfy the following conditions:
\begin{equation}\label{eq49}
(\lambda_0^h, p_{1,1}^h,..., p_{h+1,n}^h, \lambda^h_{1,0},..., \lambda^h_{m_e,h}, \mu^h_{1,0}, ..., \mu^h_{m_i, h}) \neq 0
\end{equation}
\begin{equation}\label{eq410}
\left.
\begin{array}{r}
\lambda^h_0 \geq 0, (\forall t \in \{1,...,h+1 \}, \forall \alpha \in \{1,...,n \}, p^h_{t, \alpha} \geq 0)\\
(\forall t \in \{0,...,h \}, \forall k \in \{1,...,m_i \}, \mu^h_{t,j} \geq 0)
\end{array}
\right\}
\end{equation}
\begin{equation}\label{eq411}
\forall t \in \{0,...,h \}, \forall \alpha \in \{1,...,n \}, p_{t+1, \alpha}^h \cdot F^{\alpha}_t({\bf \hat{z}}) = 0
\end{equation}
\begin{equation}\label{eq412}
\forall t \in \{0,...,h \}, \forall k \in \{1,...,m_i \}, \mu^h_{t,k} \cdot G^h_t({\bf \hat{z}}) = 0
\end{equation}
\begin{equation}\label{eq413}
\left.
\begin{array}{r}
\lambda^h_0 D \Phi({\bf \hat{z}}) + \sum_{\alpha = 1}^n p^h_{t, \alpha} D_GF^{\alpha}_t({\bf \hat{z}})\\
+ \sum_{j=1}^{m_e} \lambda^h_{t,j} D E^j_t({\bf \hat{z}}) + \sum_{k=1}^{m_i} \mu^h_{t,k} D_G G^k_t({\bf \hat{z}}) = 0.
\end{array}
\right\}
\end{equation}
Note that (\ref{eq411}) is translated by
\begin{equation}\label{eq414}
\forall t \in \{0,...,h \}, \forall \alpha \in \{1,...,n \}, p_{t+1, \alpha}^h \cdot (f^{\alpha}_t(\hat{x}_t, \hat{u}_t) - \hat{x}^{\alpha}_{t+1}) = 0.
\end{equation}
The condition (\ref{eq412}) is translated by
\begin{equation}\label{eq415}
\forall t \in \{0,...,h \}, \forall k \in \{ 1,...,mi \}, \mu_{t, k}^h \cdot g^h_t(\hat{u}_t) = 0.
\end{equation}
From (\ref{eq413}), using the partial differentiations with respect to $x_t$ we obtain, for all $\delta x_t \in \R^n$,
$$\lambda^h_0 D_1 \phi_t(\hat{x}_t, \hat{u}_t) \delta x_t + \sum_{\alpha = 1}^n p^h_{t+1, \alpha} D_{G,1} f^{\alpha}_t(\hat{x}_t, \hat{u}_t) \delta x_t - \sum_{\alpha = 1}^n p^h_{t+1, \alpha} \cdot \delta x_t +0+0 = 0$$
which implies, denoting by $p^h_t$ the elements of $\R^{n*}$ whose the coordinates are the $p^h_{t, \alpha}$, we obtain the following relation.
\begin{equation}\label{eq416}
\lambda^h_0 D_1 \phi_t(\hat{x}_t, \hat{u}_t) + p^h_{t+1} \circ D_{G,1}f_t (\hat{x}_t, \hat{u}_t)= p^h_t.
\end{equation}
From (\ref{eq413}) using the partial differentiations with respect to$u_t$ we obtain the following relation.
\begin{equation}\label{eq417}
\left.
\begin{array}{r}
\lambda^h_0 D_2 \phi_t(\hat{x}_t, \hat{u}_t) + p^h_{t+1} \circ D_{G,2}f_t (\hat{x}_t, \hat{u}_t)\\
+ \sum_{j=1}^{m_e} \lambda^h_{t,j} De^j_t(\hat{u}_t) + \sum_{k=1}^{m_i} \mu^h_{t,k} D_G g^k_t(\hat{u}_t) = 0.
\end{array}
\right\}
\end{equation}
Using (ix) and working as in the proof in the proof of Theorem \ref{th31} or using (x) and working as in the proof of Theorem \ref{th32}, from (\ref{eq416}) we obtain the following condition.
\begin{equation}\label{eq418}
(\lambda^h_0, p^h_1) = (0,0) \Longrightarrow (\lambda^h_0, p^h_1, ..., p^h_{h+1}) = (0,0, ...,0).
\end{equation}
If $(\lambda^h_0, p^h_1) = (0,0)$, using (\ref{eq418}), (\ref{eq417}) implies 
$$\sum_{j=1}^{m_e} \lambda^h_{t,j} De^j_t(\hat{u}_t) + \sum_{k=1}^{m_i} \mu^h_{t,k} D_G g^k_t(\hat{u}_t) = 0,$$
and using (\ref{eq415}) we obtain that $\mu^h_{t,k} = 0$ if $k \notin I^s_t$, and so we obtain the following relation $\sum_{j=1}^{m_e} \lambda^h_{t,j} De^j_t(\hat{u}_t) + \sum_{k \in I^s_t} \mu^h_{t,k} D_G g^k_t(\hat{u}_t) = 0$. Then using (vii) and Lemma \ref{lem45} we obtain that $\mu^h_{t,k} = 0$ for all $k \in I^s_t$, and consequently we have $\mu^h_{t,k} = 0$ for all $k \in \{1,...,m_i \}$. Therefore we have $\sum_{j=1}^{m_e} \lambda^h_{t,j} De^j_t(\hat{u}_t) = 0$.  Using (viii) we obtain $\lambda^h_{t,j} = 0$ for all $j \in \{1,...,m_e \}$. And so we have proven that  $(\lambda^h_0, p^h_1) = (0,0)$ implies $(\lambda_0^h, p_{1,1}^h,..., p_{h+1,n}^h, \lambda^h_{1,0},..., \lambda^h_{m_e,h}, \mu^h_{1,0}, ..., \mu^h_{m_i, h}) = (0, ...,0)$ which is a contradiction with (\ref{eq49}). And so we have proven the following condition.
\begin{equation}\label{eq419}
(\lambda^h_0, p^h_1) \neq (0,0).
\end{equation}
From (\ref{eq419}) under (ix) proceeding as in the proof of Theorem \ref{th31} or, under (x) proceeding as in the proof of Theorem \ref{th32} we obtain the existence of an increasing function $r : \N_* \rightarrow \N_*$, of $\lambda_0 \in \R$ and of $(p_t)_{t \in \N_*} \in (\R^{n*})^{\N_*}$ such that
\begin{equation}\label{eq420}
 \lim_{h \rightarrow + \infty} \lambda_0^{r(h)} = \lambda_0, (\forall t \in \N_*, \lim_{h \rightarrow + \infty}p_t^{r(h)} = p_t), (\lambda_0, p_1) \neq (0,0).
\end{equation}
From (\ref{eq420}) we see that the sequences $(\lambda_0^{r(h)})_{h \in \N_*}$ and $(p_t^{r(h)})_{h \in \N_*}$ are bounded and then, using (\ref{eq417}), we deduce that the sequence 
\[
\begin{array}{l}
(\sum_{j=1}^{m_e} \lambda^{r(h)}_{t,j} D e^j_t(\hat{u}_t) + \sum_{k=1}^{m_i} \mu^{r(h)}_{t,k} D_G g^k_t(\hat{u}_t))_{h \in \N_*} = \\
(\sum_{j=1}^{m_e} \lambda^{r(h)}_{t,j} D e^j_t(\hat{u}_t) + \sum_{k \in I^s_t} \mu^{r(h)}_{t,k} D_G g^k_t(\hat{u}_t))_{h \in \N_*}
\end{array}
\]
is bounded for all $t \in \N$. Using (vii), (viii) and Lemma \ref{lem46} we can assert that there exist an increasing function $r_1 : \N_* \rightarrow \N_*$, $\lambda_{t,j} \in \R$ (for all $t \in \N$ and for all $j \in \{1,..., m_e \}$), $\mu_{t,k} \in \R$ (for all $t \in \N$ and for all $k \in \{ 1,..., m_i \}$) such that 
\begin{equation}\label{eq421}
\lim_{h \rightarrow + \infty} \lambda^{r \circ r_1(h)}_{t,j} = \lambda_{t,j}, \lim_{h \rightarrow + \infty} \mu^{r \circ r_1(h)}_{t,k} = \mu_{t,k}. 
\end{equation}
\vskip1mm
Finally (\ref{eq420}) implies (NN), (\ref{eq420}), (\ref{eq421}) and (\ref{eq410}) imply (Si), (\ref{eq420}), (\ref{eq421}), (\ref{eq414}) and (\ref{eq415}) imply (S${\ell}$), 
(\ref{eq420}) and (\ref{eq416}) imply (AE), and (\ref{eq420}), (\ref{eq421}) and (\ref{eq417}) imply (WM).
\end{proof}
\begin{theorem}\label{th48}
Let $(\underline{\hat{x}}, \underline{\hat{u}})$ be a solution of $(P^j_e)$ where $j \in \{1,2,3 \}$ and where the sets $U_t$ are defined in (\ref{eq47}). We assume that the following assumptions are fulfilled for all $t \in \N$.
\begin{itemize}
\item[(i)] $\phi_t$ is Fr\'echet differentiable at $(\hat{x}_t, \hat{u}_t)$.
\item[(ii)] $f_t$ is continuous on a neighborhood of $(\hat{x}_t, \hat{u}_t)$ and Fr\'echet differentiable at $(\hat{x}_t, \hat{u}_t)$. 
\item[(iii)] For all $k \in \{ 1, ..., m_i \}$, $g^k_t$ is Fr\'echet differentiable at $\hat{u}_t$ when $g^k_t(\hat{u}_t) = 0$.
\item[(iv)] For all $k \in \{ 1, ..., m_i \}$, $g^k_t$ is lower semicontinuous and G\^ateaux differentiable at $\hat{u}_t$ when $g^k_t(\hat{u}_t) > 0$.
\item[(v)] For all $j \in \{ 1, ..., m_e \}$, $e^j_t$ is continuous on a neighborhood of $\hat{u}_t$ and Fr\'echet differentiable at $\hat{u}_t$.
\item[(vi)] $span \{ De^j_t(\hat{u}_t) : j \in \{ 1,..., m_e \} \} \cap co \{ D_Gg^k_t(\hat{u}_t) : k \in I^s_t \} = \emptyset$, where $I^s_t := \{ k \in \{1,...,m_i \} : g^k_t(\hat{u}_t) = 0 \}$.
\item[(vii)] $De_t^1(\hat{u}_t)$, ..., $De^{m_e}_t(\hat{u}_t)$ are linearly independent.
\item[(viii)] For all $t \in \N_*$, $D_{G,1}f_t(\hat{x}_t, \hat{u}_t)$ is invertible.

\end{itemize}
Then under (i-viii) there exist $\lambda_0 \in \R$, $(p_t)_{t \in \N_*} \in (R^{n*})^{\N_*}$, $(\lambda_{t,1})_{t \in \N} \in \R^{\N}$,..., $(\lambda_{t, m_e})_{ t \in \N} \in \R^{\N}$, $(\mu_{t,1})_{t \in \N} \in \R^{\N}$, ..., $(\mu_{t,m_i})_{ t \in \N}\in \R^{\N}$ which satisfy the following conditions.
\begin{itemize}
\item[(NN)] $(\lambda_0, p_1) \neq (0,0)$.
\item[(Si)] $\lambda_0 \geq 0$,  $\mu_{k,t} \geq 0$ for all $t \in \N$ and for all $k \in \{1,...,m_i \}$.
\item[(S${\ell}$)] For all $t \in \N$, for all $k \in \{1,...,m_i \}$, $\mu_{k,t} \cdot g_t^k(\hat{u}_t) = 0$.
\item[(AE)] For all $t \in \N$, $p_t = p_{t+1} \circ D_{G,1} f_t(\hat{x}_t, \hat{u}_t) + \lambda_0 D_1 \phi_t (\hat{x}_t, \hat{u}_t)$.
\item[(WM)] For all $t \in \N$, \\
$p_{t+1} \circ D_{G,2} f_t(\hat{x}_t, \hat{u}_t) + \lambda_0 D_2 \phi_t (\hat{x}_t, \hat{u}_t) + \displaystyle\sum_{j=1}^{m_e}\lambda_{t,j} De^j(\hat{u}_t) + \displaystyle\sum_{k=1}^{m_i} \mu_{t,k} D_G g^k_t(\hat{u}_t) = 0.$
\end{itemize}
\end{theorem}
The proof of this theorem is similar to the this one of Theorem \ref{th47}. The difference is the replacement of inequality constraints by equality constraints in the problem issued from the reduction to finite horizon, the consequence of this difference is the lost of the sign of the adjoint variable $p_t$.


\begin{thebibliography}{00}
%
\bibitem{Bl1} J. Blot, Infinite-horizon Pontryagin principle without invertibility, J. Nonlinear Convex Anal., {\bf 10}(2), 157-176, 2009.
%
\bibitem{Bl2} J. Blot, A Pontryagin principle for infinite-horizon problems under constraints, Dyn. Contin. Discr. Impul. Syst., Series B: Appl. Algor., {\bf 19}, 267-275, 2012.
%
\bibitem{Bl3} J. Blot, Infinite-horizon discrete-time Pontryagin principles via results of Michel, in {\it Proceedings of the Haifa Workshop on Optimization and Related Topics}, ed. by S. Reich, A.J. Zaslaski, Contemporary Mathematics, vol. 568, 41-51, 2012.
%
\bibitem{Bl4} J. Blot, On the multiplier rules, Optimization, DOI: 10.1080/02331934.2015.1113531, 2015.
%
\bibitem{BCh} J. Blot \& H. Chebbi, Discrete time Pontryagin principle in infinite horizon, J. Math. Anal. Appl., {\bf 246}, 265-279, 2000.
%
\bibitem{BH} J. Blot \&  N. Hayek, {\sf Infinite-horizon optimal control in the discrete-time framework}, Springer, New York, 2014.
%
\bibitem{Fl} T.M. Flett, {\sf Differential analysis}, Cambridge University Press, Cambridge UK, 1980.
%
\end{thebibliography}
\end{document}